\documentclass[12pt,a4paper]{amsart}
\usepackage{amsmath}
\usepackage{amssymb}
\usepackage{latexsym}
\usepackage{xypic}
\input xy
\xyoption{all}

\newtheorem{theorem}{Theorem}[section]

\newtheorem{corollary}[theorem]{Corollary}

\newtheorem{exmp}[theorem]{Example}
\newtheorem{exmps}[theorem]{Examples}
\newtheorem{rem}[theorem]{Remark}

\newenvironment{remark}{\begin{rem}\rm}{\end{rem}\rm}

\newcommand{\qqed}{\hspace*{\fill}$\Box$}

\newcommand{\beeq}[1]{\begin{eqnarray}\label{#1}}
\newcommand{\eneq}{\end{eqnarray}}

\newcommand{\Db}{{\rm D}^{\rm b}}

\newcommand{\ZZ}{\mathbb{Z}}

\newcommand{\PP}{\mathbb{P}}

\newcommand{\CH}{{\rm CH}}

\newcommand{\NS}{{\rm NS}}
\newcommand{\Aut}{{\rm Aut}}

\newcommand{\id}{{\rm id}}

\newcommand{\rk}{{\rm rk}}
\newcommand{\ord}{{\rm ord}}

\newcommand{\verylongarrow}[1]{\hbox to #1{\rightarrowfill}}
\newcommand{\congpf}{\xymatrix@1@=15pt{\ar[r]^-\sim&}}
\renewcommand{\to}{\xymatrix@1@=15pt{\ar[r]&}}

\begin{document}

\title[Symplectic automorphisms of K3 surfaces of arbitrary order]{Symplectic automorphisms of K3 surfaces of arbitrary order}

\author[D.\ Huybrechts]{D.\ Huybrechts}
\address{{\rm Daniel Huybrechts}\\
Mathematisches Institut, Universit{\"a}t Bonn\\
 Endenicher Allee  60\\
 53115 Bonn\\
 Germany}
\email{huybrech@math.uni-bonn.de}

%\thanks{This work was supported by the SFB/TR 45 `Periods, Moduli Spaces and Arithmetic of Algebraic Varieties' of the DFG (German Research Foundation).}

\begin{abstract} It is observed that the existing results in \cite{HuyMSRI} and \cite{VoisinSympl} suffice to prove 
in complete generality that symplectic automorphisms of finite order of
a K3 surface $X$ act as $\id$ on the Chow group $\CH^2(X)$ of zero-cycles.
\end{abstract}
 \maketitle
 
In \cite{VoisinSympl}, Claire Voisin recently proved that 
symplectic involutions of a K3 surface $X$ act trivially on the Chow group
$\CH^2(X)$ of $0$-dimensional cycles. We shall explain that this
can be combined with a result in  \cite{HuyMSRI},  which in turn is based on
techniques involving the derived category of coherent sheaves $\Db(X)$ developed
in \cite{Huy}, to prove the following

\begin{theorem}\label{thm:main}
Let $f:X\congpf X$ be a symplectic automorphism of finite order of a complex
projective K3 surface $X$. Then $f^*={\rm id}$ on $\CH^2(X)$.
\end{theorem}

As explained at various places before, the result can be seen as evidence for the much more general
philosophy of Boch and Beilinson on filtrations of Chow groups and their relation
to  Hodge theory, see e.g.\ \cite{HuyMSRI, VoisinSympl}. 
The particular case of automorphisms $f:X\congpf X$ of a K3 surface 
with $f^*={\rm id}$ on $H^{2,0}(X)$ has been addressed explicitly in \cite{Huy}
in the context of autoequivalences of the derived category $\Db({\rm Coh}(X))$.
If the induced action $f^*:H^*(X,\ZZ)\congpf H^*(X,\ZZ)$ (on the full cohomology!) of
a symplectic automorphism $f:X\congpf X$ coincides with the action induced by
a composition $\prod T_{E_i}$ of spherical twists $T_{E_i}:\Db({\rm Coh}(X))\congpf \Db({\rm Coh}(X))$,
it was shown  that the action $f^*:\CH^2(X)\congpf\CH^2(X)$ is trivial.
The possibility of using this to prove $f^*=\id$ in general was studied in \cite{HuyMSRI}.
There,  a condition on the $2$-rank and the $3$-rank of
the N\'eron--Severi lattice $\NS(X)\subset H^2(X,\ZZ)$ was formulated, that
ensures that indeed  for any symplectic automorphism $f^*=\prod T_{E_i}$ on $H^*(X,\ZZ)$
for some spherical twists.
It was remarked that the condition on $\NS(X)$ is verified for a dense subset of K3 surfaces
$(X,f)$  with a symplectic involution, but that it definitely fails in general.
Since the case of symplectic  involutions seemed the most approachable, this attack  was
abandoned at the time. Instead, we  tried in \cite{HuyKem} to approach the case of involutions
by more geometric arguments and succeeded in proving the theorem in one third of all cases.
The same techniques (exis\-tence of nodal elliptic curves via deformation theory of stable maps)
apply more generally to symplectic automorphisms of higher order but yield the result
only for some components of the moduli space of $(X,f)$.

Cases of symplectic involutions on K3 surfaces of low degree have also been studied 
in \cite{GT,Ped,Voi,Voisin2}, but the best result has been obtained by Voisin in \cite{VoisinSympl}.
Instead of working with invariant elliptic curves as in \cite{HuyKem}, she uses Prym varieties to deduce finite
dimensionality (in the sense of Roitman) of the anti-invariant part of $\CH^2(X)$ which
is enough to prove the result. Unfortunately, the same idea seems to fail for symplectic automorphisms
of  order $>2$, as the dimensions of the linear systems and the Prym varieties do not match up.

It is curious to note, that for higher order automorphisms, where the beautiful geometric arguments of \cite{VoisinSympl} seemed to fail, the derived and less geometric techniques of \cite{HuyMSRI} work
without any further work (although they give only very partial results for
involutions). In fact, the point of this note is to remark that the 
case of symplectic automorphisms of higher order could have been settled already in \cite{HuyMSRI}.
If only I had appreciated the very useful classification results of \cite{GS} more and had put them to better use.

{\bf Acknowledgments:} Thanks to  A.\ Garbagnati, B.\ van Geemen, and C.\ Voisin for comments on a first version of this note.
I am particularly grateful to A.\ Sarti for comments on the case of order $4$ and $8$.

 \section{Proof}
 
 Let $X$ be a complex K3 surface and let $f:X\congpf X$ be a symplectic automorphism
 of prime order $p=\ord(f)$. Thus, the induced action of $f$ on $H^2(X,\ZZ)$ is the identity on
 $H^{2,0}(X)$ and hence on the transcendental lattice $T(X)$. As was shown by Nikulin in
 \cite{Nik}
 there are only the following possibilities: $p=2,3,5,7$. Of course, it suffices to settle Theorem
 \ref{thm:main}
 in these cases.
 
 \subsection{} If $H^2(X,\ZZ)$ is considered
 as an abstract lattice, then the $f$-invariant part and its orthogonal complement have been studied
 in great detail by Nikulin \cite{Nik}, Morrison \cite{Mor}, van Geemen and Sarti
 \cite{vGS}, and Garbagnati and Sarti \cite{GS}.
 As it turns out,  the invariant part  $H^2(X,\ZZ)^f$ and its orthogonal complement $(H^2(X,\ZZ)^f)^\perp$
only depend on the order $p$ and can be described abstractly. 

We shall briefly review the facts relevant for our purpose. The orthogonal complement
of $H^2(X,\ZZ)^f$ will simply be denoted $$\Omega_p:=(H^2(X,\ZZ)^f)^\perp.$$
Recall that $H^2(X,\ZZ)$ is isomorphic
to $U\oplus U\oplus U\oplus E_8(-1)\oplus E_8(-1)$.

i) $p=2$. Then $H^2(X,\ZZ)^f\cong U\oplus U\oplus U\oplus E_8(-2)$ and 
$\Omega_2\cong E_8(-2)$, which is a lattice of rank $8$ and discriminant $2^8$,
cf.\ \cite[Sec.\ 1.3]{vGS} or \cite[Thm.\ 4.1]{GS}.

ii) $p=3$. Then $H^2(X,\ZZ)^f\cong U\oplus U(3)\oplus U(3)\oplus A_2\oplus A_2$ and 
$\Omega_3$ is a lattice of rank $12$ and discriminant $3^6$, cf.\ \cite[Thm.\ 4.1, Prop.\ Ê4.2]{GS}.

iii) $p=5$. Then $H^2(X,\ZZ)^f\cong U\oplus U(5)\oplus U(5)$ and $\Omega_5$ is a lattice of rank
$16$ and discriminant $5^4$, cf.\ \cite[Thm.\ 4.1, Prop.\ 4.4]{GS}.

iv) $p=7$. Then $H^2(X,\ZZ)^f \cong U(7)\oplus \left(\begin{matrix}4&1\\1&2\end{matrix}\right)$
and $\Omega_7$ is a lattice of rank $18$ and discriminant $7^3$, cf.\ \cite[Thm.\ 4.1, Prop.\ 4.6]{GS}.

In \cite{GS} one finds much more interesting information, e.g.\ on the discriminant form and the action of $f$ on it,
but this is not needed here.

As usual, $U$ denotes the hyperbolic plane $\left(\begin{matrix}0&1\\1&0\end{matrix}\right)$
and $E_8$ is the unique even unimodular positive lattice of rank $8$. For any lattice $\Lambda$, the 
twist $\Lambda(n)$ by an integer $n$ denotes the lattice that is given by $n(~~,~~)_\Lambda$, where
$(~~,~~)_\Lambda$ is the pairing on $\Lambda$. 

\subsection{} Recall that  for a prime number $q$, the $q$-rank $\rk_q(\Lambda)$
of a lattice $\Lambda$ is the maximal rank of a sublattice $\Lambda' \subset\Lambda$
whose discriminant is not divisible by $q$. For the application we have in mind, only the cases
$q=2$ and $q=3$ will play a role.

As a consequence of the results of \cite{GS} one immediately finds the following

\begin{corollary}
i) If $p=5$, then   $\rk_2(\Omega_5)=\rk_3(\Omega_5)=16$. 

ii) If $p=7$, then  $\rk_2(\Omega_7)=\rk_3(\Omega_7)=18$.\qqed
\end{corollary}

For $p=3$, the lattics $\Omega_3$ has maximal $2$-rank, but its $3$-rank is strictly smaller than
$\rk(\Omega_3)=12$, but the following weaker result will suffice.

\begin{corollary}
If $p=3$, then $\rk_2(\Omega_3)=12$ and $\rk_3(\Omega_3)\geq3$.
\end{corollary}

\begin{proof}
The first assertion is obvious and the second follows from the explicit description of the intersection
matrix of $\Omega_3$ given in \cite[Prop.\ 4.2]{GS} in terms of a basis $e_1,\ldots,e_{12}$.
We refrain from copying the matrix here, but a quick look
shows that $e_1,e_7$ and $e_{11}$ (or, alternatively, $e_1,e_5$, and $e_{11}$) span a sublattice $\Lambda'$
which modulo $3$ has a diagonal intersection matrix $diag(2,2,2)$, 
the discriminant of which is clearly not divisible by $3$.
\end{proof}

\begin{remark}
For $p=2$, the $3$-rank of $\Omega_2$ is again maximal, but the $2$-rank is zero. This explains why 
 \cite[Thm.\  6.5]{HuyMSRI} definitely does not apply to the generic symplectic involution.
\end{remark}

\subsection{} For the convenience of the reader, we recall 
the following result, which is \cite[Thm.\ 6.5]{HuyMSRI}.

\begin{theorem}\label{thm:Kneserappl}
Suppose $\rk_2(\NS(X))\geq 4$ and $\rk_3(\NS(X))\geq3$. Then any
symplectomorphism $f\in\Aut(X)$ acts trivially on $\CH^2(X)$.
\end{theorem}

This theorem uses a result of Kneser, generalizing results of Wall for unimodular lattices, that shows that
under the assumptions on $\NS(X)$  every orthogonal transformation of $\NS(X)\oplus H^0(X,\ZZ)\oplus H^4(X,\ZZ)$ with trivial spinor norm and discriminant is a composition of reflections. But reflections can always be lifted to spherical twists. Note that the spherical twists $T_E$ that occur in $f^*=\prod T_{E_i}^*$ will typically  not respect $H^2$.

\begin{corollary}\label{cor:357}
Let $f:X\congpf X$ be a symplectic automorphism of a projective K3 surface $X$ and suppose that the order of $f$ is
$p=3,5$, or $7$. Then the induced action of $f$ on $\CH^2(X)$ is trivial. \qqed
\end{corollary}

In order to conclude the proof of Theorem \ref{thm:main}, one first invokes Nikulin's result that
a symplectic automorphism $f:X\congpf X$ of a K3 surface $X$ has order $\ord(f)\leq 8$, see \cite{Nik}.
Voisin's recent result in \cite{VoisinSympl} settles the case $\ord(f)=2$ and, according to Corollary
\ref{cor:357}, only the cases $\ord(f)=4,6,8$ remain to be proven.

 The case $\ord(f)=6$ follows
immediately, as $\ZZ/6\ZZ\cong\ZZ/3\ZZ\times\ZZ/2\ZZ$ and Corollary \ref{cor:357} applies to the generators of the two factors. For $\ord(f)=4,8$ one applies \cite{VoisinSympl} repeatedly.
Let us demonstrate this for $\ord(f)=4$, the case $\ord(f)=8$ is similar. Consider the quotient
$X\to \bar X:=X/\langle f^2\rangle$. Since $f^2$ is a symplectic involution, one knows
by \cite{VoisinSympl}  that $[x]=[f^2(x)]\in\CH^2(X)$. Now apply \cite{VoisinSympl}
again to   $\bar X$, which is a (singular) K3 surface, and the symplectic automorphism $\bar f:\bar X\congpf \bar X$,  which is of order $2$. Hence, the images $\bar x,\bar f(\bar x)\in\bar X$ are rationally equivalent. Pulled-back to $X$, this shows that $[x]+[f^2(x)]=[f(x)]+[f^3(x)]$ and hence $2[x]=2[f(x)]$,
but $\CH^2(X)$ is torsion free.  (Since it is enough to prove $[x]=[f(x)]$ for generic $x\in X$,  the singularities of $\bar X$ are of little importance. Just apply the  arguments to its minimal desingularization.)

In fact, as A.\ Sarti has pointed out to the author, the case $\ord(f)=8$ could also be deduced
from Theorem \ref{thm:Kneserappl}. Indeed, in this case $\Omega_G$ is of rank $18$ (see \cite[Prop.\ 5.1]{GS2}) and hence $\rk(\NS(X))\geq19$. By Cororllary 2.6, i) in \cite{Mor} one knows that $T(X)\subset U^{\oplus 3}$ and by Theorem 6.3 of the same paper this implies $E_8(-1)^{\oplus 2}\subset \NS(X)$, which allows one to apply Theorem \ref{thm:Kneserappl}. In fact, the same reasoning also applies to
$\ord(f)=7$. 
 %%%%%%%%%%%%%%%%%%%%%%%%%%%%%%%%%%%%% 

\end{document}